\documentclass[11pt,leqno]{article}
\usepackage[latin1]{inputenc}
\usepackage{amssymb, amsthm, amsmath}
\usepackage{amsfonts}
\usepackage{mathrsfs}
\usepackage{graphicx}
\usepackage{latexsym}
\usepackage{comment}
\usepackage{setspace}
\usepackage[left=2.50cm, right=2.50cm, top=3.00cm, bottom=3.00cm]{geometry}
\newtheorem{definition}{\bf Definition}[section]
\newtheorem{theorem}[definition]{\bf Theorem}
\newtheorem{lemma}[definition]{\bf Lemma}

\allowdisplaybreaks[0]

\makeatletter

\@addtoreset{equation}{section}
\makeatother
\begin{document}
%\label{page:t}
%\thispagestyle{plain}
\title{ Bounds for global solutions of a reaction diffusion system\\ with the Robin boundary conditions\\[3mm]}
%%
%%\dedication{Dedicated to the memory of Professor Riichi Iino}
%%
\author{{\large Kosuke Kita}\\[3mm]
	Graduate School of Advanced Science and Engineering, \\ Waseda University, 3-4-1 Okubo Shinjuku-ku, Tokyo, 169-8555, JAPAN\\[5mm]
		{\large Mitsuharu \^{O}tani}\\[3mm]
	Department of Applied Physics, School of Science and Engineering, \\ Waseda University, 3-4-1 Okubo Shinjuku-ku, Tokyo, 169-8555, JAPAN\\[5mm]}
%% \footnotemark[1]}
\date{}
%\affiliation{Major in Pure and Applied Physics, Graduate School of Advanced Science and Engineering, \\ Waseda University, 3-4-1 Okubo Shinjuku-ku, Tokyo, 169-8555, JAPAN}
%\email{kou5619@asagi.waseda.jp}
%
%\sauthor{Mitsuharu \^{O}tani \footnotemark[2]}
%\saffiliation{Department of Applied Physics, School of Science and Engineering, \\ Waseda University, 3-4-1 Okubo Shinjuku-ku, Tokyo, 169-8555, JAPAN}
%\semail{otani@waseda.jp}
%
%\tauthor{Hiroki Sakamoto}
%\taffiliation{Hitachi-GE Nuclear Energy, Ltd.\\ 3-1-1, Saiwai-cho, Hitachi-shi, Ibaraki-ken, 317-0073, JAPAN}
%\temail{hiroki.sakamoto.ec@hitachi.com}
%
%%\fauthor{Name of 4th author}
%%\faffiliation{1st line of affiliation \\ 2nd line of affiliation}
%%\femail{e-mail address of 4th author}
%%Please do not use \date 
%\footnotetext[1]{}
\footnotetext[2]{Partly supported by the Grant-in-Aid 
	for Scientific Research, \# 15K13451 and \# 18K03382,
	the Ministry of Education, Culture, Sports, Science and Technology, Japan.}
\footnotetext[3]{e-mail : kou5619@asagi.waseda.jp}
\maketitle
\noindent
{\bf Abstract.}
%%%%%%%%%%%%%%%%%%%%%%%%%%%%%%%%%%%%%%%%%%%%%%%%%%%%%%%%%%%%%%%%%%%%%%%%%%%%%%%
%%%%%%   Abstract               %%%%%%%%%%%%%%%%%%%%%%%%%%%%%%%%%%%%%%%%%%%%%%%
%%%%%%%%%%%%%%%%%%%%%%%%%%%%%%%%%%%%%%%%%%%%%%%%%%%%%%%%%%%%%%%%%%%%%%%%%%%%%%%
In this paper, we are concerned with the large-time behavior of solutions of a reaction diffusion system arising from a nuclear reactor model with the Robin boundary conditions.
It is shown that global solutions of this system are uniformly bounded in a suitable norm with respect to time. \\

%\newpage

%%%%%%%%%%%%%%%%%%%%%%%%%%%%%%%%%%%%%%%%%%%%%%%%%%%%%%%%%%%%%%%%%%%%%%%%%%%%%%%%%
%%%%%%%%%%   Introduction               %%%%%%%%%%%%%%%%%%%%%%%%%%%%%%%%%%%%%%%%%
%%%%%%%%%%%%%%%%%%%%%%%%%%%%%%%%%%%%%%%%%%%%%%%%%%%%%%%%%%%%%%%%%%%%%%%%%%%%%%%%%
\section{Introduction}

\hspace{\parindent}We consider the asymptotic behavior of global solutions of the initial boundary value problem for a reaction diffusion system : 
\begin{equation}
\label{1.1}
\left\{
\begin{aligned}
& \partial_t u_1-\Delta u_1 = u_1u_2 - bu_1, &&\hspace{7mm} t>0,~x\in\Omega,\\
& \partial_t u_2-\Delta u_2 = au_1, &&\hspace{7mm} t>0,~x\in\Omega,\\
& \partial_{\nu}u_1 + \alpha u_1 = \partial_{\nu}u_2 + \beta u_2 = 0, &&\hspace{7mm} t>0,~x\in\partial\Omega,\\
& u_1(0,x)=u_{10}(x)\ge0,~u_2(0,x)=u_{20}(x)\ge0, &&\hspace{7mm}x\in\Omega.
\end{aligned}
\right.
\end{equation}
Here \(\Omega\) is a bounded domain in \(\mathbb{R}^N\) with smooth boundary \(\partial\Omega\), and \(\nu\) denotes the unit outward normal vector on \(\partial\Omega\). 
Furthermore \(u_1\), \(u_2\) are real-valued unknown functions and \(a\), \(b\) are given positive constants.
We also assume \(\alpha\ge 0\) and \(\beta>0\).
This problem is introduced in 1968 by Kastenberg and Chambr\'e \(\cite{KC1}\) for the purpose to give mathematical model of a nuclear reactor, where \(u_1\) represents the neutron flux and \(u_2\) represents the fuel temperature.

This model is studied by many authors under various (linear) boundary conditions, see, e.g., \(\cite{C1}\), \(\cite{C2}\), \(\cite{GW1}\), \(\cite{GW2}\), \(\cite{JLZ}\), \(\cite{R}\) and \(\cite{Y}\). 
They investigated the existence of positive steady-state solutions and the asymptotic behavior of solutions.
In our previous work \(\cite{KO}\), we also studied the initial-boundary value problem for this system with nonlinear boundary conditions:
\begin{equation}
\label{1.2}
\left\{
\begin{aligned}
& \partial_t u_1-\Delta u_1 = u_1u_2 - bu_1, &&\hspace{7mm} t>0,~x\in\Omega,\\
& \partial_t u_2-\Delta u_2 = au_1, &&\hspace{7mm} t>0,~x\in\Omega,\\
& \partial_{\nu}u_1 + \alpha u_1 = \partial_{\nu}u_2 + \beta |u_2|^{\gamma-2}u_2 = 0, &&\hspace{7mm} t>0,~x\in\partial\Omega,\\
& u_1(0,x)=u_{10}(x)\ge0,~u_2(0,x)=u_{20}(x)\ge0, &&\hspace{7mm}x\in\Omega,
\end{aligned}
\right.
\end{equation}
where \(\gamma\ge 2\).
We showed the existence and the ordered uniqueness of positive stationary solution for \(N\in[1,5]\).
For nonstationary problem, we proved that any positive stationary solution plays a role of threshold to separate global solutions and finite time blowing-up solutions.
%%%%%%%%%%%
More precisely, if the initial data is less than or equal to positive stationary solutions, then solutions of \(\eqref{1.2}\) exists globally and tends to zero as \(t\to\infty\), 
and if the initial data is strictly larger than positive stationary solutions, then 
solutions of \(\eqref{1.2}\) blow up in finite time.
For general initial data, however, this result does not say anything about the asymptotic behavior of global solutions.
When we assume that solutions exist globally, it is natural to ask whether global solutions blow up at \(\infty\) or not. 
We here restrict ourselves to the case where \(\gamma=2\), for the technical reason.
Bounds for global solutions of this system with the homogeneous Dirichlet boundary conditions is already studied by Quittner \(\cite{Q1}\) for the case where \(N=2\).
This strong restriction on \(N\) arises from applying Hardy type inequality (see \(\cite{BT}\)). 
%%%%%%
As for the Robin boundary conditions, by making use of the good properties of the first eigenfunction of Laplacian with Robin boundary conditions, we can discuss the case where \(N=2,3\). 

This kind of problem is well known for the scalar problem :
\begin{equation}
\label{1.3}
\left\{
\begin{aligned}
& \partial_t u(t,x)-\Delta u(t,x) = f(u(t,x)), &&\hspace{7mm} t>0,~x\in\Omega,\\
& u(t,x) = 0, &&\hspace{7mm} t>0,~x\in\partial\Omega,\\
& u(0,x)=u_{0}(x), && \hspace{7mm}x\in\Omega.
\end{aligned}
\right.
\end{equation}
For simplicity, assume that \(f(u)=|u|^{p-2}u\) and \(p\) is Sobolev subcritical, that is, \(p\in(2,p_S)\), where
\(p_S\) is the Sobolev critical exponent defined by \(p_S=\infty\) for \(N=1,2\) ; \(p_S=\frac{2N}{N-2}\) for \(N=3\).
%%%%%%%%%%%%%%%%%%%%%%%%
The boundedness of global solutions of \eqref{1.3} was first discussed by \(\cite{O1,O2}\) in the abstract setting of the form \( u_t + \partial\varphi^1(u)-\partial\varphi^2(u)=0\) in \(L^2(\Omega)\).
Here \(\partial\varphi^i\) are subdifferentials of lower semi-continuous convex and homogeneous functionals \(\varphi^i\) (\(i=1,2\)) on \(L^2(\Omega)\), where 
 it is shown that every global solution of \(\eqref{1.3}\) is uniformly bounded in \(H^1_0(\Omega)\) with respect to time.
 %%%%%%%%%%%%%%%%%
Ni-Sacks-Tavantzis \(\cite{NST}\) studied \(\eqref{1.3}\) for the case where \(\Omega\) is convex domain and proved every positive global solution of \(\eqref{1.3}\) is uniformly bounded in \(L^{\infty}(\Omega)\) with respect to time provided that \(p\in(2,2+\frac{2}{N})\).
Furthermore they also showed that if \(p\ge p_S\), then \(\eqref{1.3}\) has a global solution which \(L^{\infty}\) norm goes to \(\infty\) as \(t\to\infty\) in the case where \(N\ge3\).
%%%%%%%%%%%%%%%%
Cazenave-Lions \(\cite{CL}\) dealt with more general nonlinear term \(f(u)\) (including \(f(u)=|u|^{p-2}u\)) and showed that every global solution allowing sing-changed solution is bounded in \(L^{\infty}(\Omega)\) uniformly in time provided that \(p\in(2,p_{CL})\), where \(p_{CL}=\infty\) when \(N=1\) ; \(p_{CL} = 2  +\frac{12}{3N-4}\) when \(N\ge2\). 
%%%%%%%%%%%%%%%%%
( Note that \(p_{CL}\le p_S\) for any \(N\in\mathbb{N}\) ).
%%%%%%%%%%%%%%
Giga removed this restriction on \(p\) in his paper \(\cite{G1}\) for positive global solutions, that is, he showed every positive global solution of \(\eqref{1.3}\) is uniformly bounded in \(L^{\infty}(\Omega)\) for any \(p\in(2,p_S)\).
%%%%%%%%%%%%%%%%%%%
 Quittner \(\cite{Q2}\) removed the restriction of the positivity of solutions, i.e., 
he proved that every global solution of \(\eqref{1.3}\) (allowing  sing-changed solution) is uniformly bounded in \(L^{\infty}(\Omega)\) for any \(p\in(2,p_S)\).
%%%%%%%%%%%%%%%%%%%%%
 
Proofs for the boundedness of global solutions of \(\eqref{1.3}\) deeply rely on  
 the fact that the energy functional \(E(u)\), defined by \(E(u)=\frac{1}{2}\int_{\Omega}|\nabla u|^2dx- \frac{1}{p}\int_{\Omega}|u|^pdx\), becomes a Lyapunov function, 
in other words, \eqref{1.3} possesses the variational structure.
In addition to that, in \(\cite{G1}\) the rescaling argument is introduced and 
in \(\cite{Q2}\) the bootstrap argument based on the interpolation and the maximal regularity
 is used. 

Unfortunately for our system, we can not apply the arguments similar to those of \(\cite{G1}\) and \(\cite{Q2}\), since \(\eqref{1.1}\) does not possess the variational structure.
 
 To cope with this difficulty, making much use of the special form of our system, 
 we first show the uniform bound for the $L^1$-norm with the positive weight 
 $\varphi_1$, the first eigenfunction of the Laplace operator with 
 the Robin boundary condition.  To derive the uniform $H^1$-bound, we rely on some energy method 
 with a special device ( see Lemma ). Furthermore by applying  
  Moser's iteration scheme such as in Nakao \cite{Na1}, we derive the uniform $L^\infty$-bound 
via $H^1$-bound. 
%%%%%%%%%%%%%%%%%%%%%%%%%%%%%%%%%%%%%%%%%%%%%%%%%%%%%%%%%%%
%%%%%   Local Existence         %%%%%%%%%%%%%%%%%%%%%%%%%%%
%%%%%%%%%%%%%%%%%%%%%%%%%%%%%%%%%%%%%%%%%%%%%%%%%%%%%%%%%%%
\section{Existence of local solutions}
\hspace{\parindent} 
Throughout this paper, we denote by \(\|\cdot\|_p\) and \(\|\cdot\|\) the norm in \(L^p(\Omega)\) \((1\le p\le\infty)\) and \(H^1(\Omega)\) respectively. 
We also simply write \(u(t)\) instead of \(u(t,\cdot)\).
 In this section, we prepare a couple of results concerning the local well-posedness. 
 The following result is proved in \cite{KO} as Theorem 3.1.
%%%%%%%%%%%%%%%%%%%%%%%%%%%%%%%%%%%%%%%%%%%%%%%%%%%%%%%%%%%%%%%%%%%%
%%%%%  Local  Existence   L^infty space   Theorem 2.1   %%%%%%%%%%%%
%%%%%%%%%%%%%%%%%%%%%%%%%%%%%%%%%%%%%%%%%%%%%%%%%%%%%%%%%%%%%%%%%%%%  
\begin{theorem}\label{LWPinfty}
	Let \((u_{10},u_{20}) \in L^{\infty}(\Omega )\times L^{\infty}(\Omega)\),  
then there exists  \(T = T(\|u_{10}\|_\infty, \|u_{20}\|_\infty)>0\) such that \eqref{1.2} possesses a unique solution 
\((u_1,u_2) \in( L^{\infty}(0,T;L^{\infty}(\Omega))\cap C([0,T];L^2(\Omega)))^2\) satisfying
	\begin{equation}\label{reg:Linfty}
		\sqrt{t}\partial_tu_1, \sqrt{t}\partial_tu_2, \sqrt{t}\Delta u_1, 
           \sqrt{t}\Delta u_2\in L^2(0,T;L^2(\Omega)).
	\end{equation}
	Furthermore, if the initial data is nonnegative, then the local solution \((u_1,u_2)\) for 
\eqref{1.2} is nonnegative.
\end{theorem}
%%%%%%%%%%%%%%%%%%%%%%%%%%%%%%%%%%%%%%%%%%%%%%%%%%%%%%%%%%%%%%%%%%%
  In order to treat the case where the data belong to $H^1(\Omega)$, 
we need to fix some abstract setting. Let $ H := L^2(\Omega) \times L^2(\Omega)$ 
and for $u=(u_1, u_2) \in H$ we put 
\begin{equation*}
  \begin{split}
  & D(\phi) := \{ ~\! u \in H ~\! ; ~\! u_1, u_2 \in H^1(\Omega), \ 
                        u_2 \in L^\gamma(\partial \Omega) ~\! \}, 
 \\[2mm]
  & \phi(u) = \begin{cases} \ 
               \displaystyle\frac{1}{2} \int_\Omega \!\!
                  ( |\nabla u_1(x)|^2 \! + b  |u_1(x)|^2 \! + |\nabla u_2(x)|^2 )  dx 
                       + \!\! \int_{\partial \Omega} \!\! \left(  \frac{\alpha}{2}|u_1(x)|^2 +
                                 \frac{\beta}{\gamma} |u_2(x)|^\gamma \right)  d\sigma 
               & \text{if} \ u \in D(\phi),
   \\[2mm]
              \ + \infty  &  \text{if} \ u \not\in D(\phi).
             \end{cases}
  \end{split}
\end{equation*}
   Then  $\phi$ is a lower semi-continuous convex function from $H$ into $[0,\infty)$ and 
its subdifferential $\partial \phi$ is given by
\begin{align*}
 & \partial \phi(u) = \{ ~\! w \in H ~\! ; ~\! w = ( - \Delta u_1 + b u_1, - \Delta u_2 ) ~\! \} 
     \quad \forall u \in D( \partial \phi), 
\\[4mm] 
  & D(\partial \phi) = \{ ~\! u = (u_1, u_2) ~\! ; ~\! u_1, u_2  \in H^2(\Omega), \   
                   \partial_{\nu}u_1 + \alpha u_1 
                     = \partial_{\nu}u_2 + \beta ~\! |u_2|^{\gamma-2}u_2 = 0 ~\! \}.             
\end{align*}
 Then we have
%%%%%%%%%%%%%%%%%%%%%%%%%%%%%%%%%%%%%%%%%%%%%%%%%%%%%%%%%%%%%%%%%%
%%%%%%%%    Local Existence H1   Theoren 2.2    %%%%%%%%%%%%%%%%%%
%%%%%%%%%%%%%%%%%%%%%%%%%%%%%%%%%%%%%%%%%%%%%%%%%%%%%%%%%%%%%%%%%%
\begin{theorem}\label{LWPH1}
	Let $N \leq 5$ and  $(u_{10},u_{20}) \in D(\phi)$,  
then there exists  \(T = T( \phi(u_0)) >0\) such that \eqref{1.2} possesses a unique solution 
\((u_1,u_2)\in( C([0,T];L^2(\Omega)))^2\) satisfying
	\begin{equation}\label{reg:H1}
		 \partial_tu_1, \partial_tu_2, \Delta u_1, 
           \Delta u_2\in L^2(0,T;L^2(\Omega)).
	\end{equation}
	Furthermore, if the initial data is nonnegative, then the local solution \((u_1,u_2)\) for 
\eqref{1.2} is nonnegative.
\end{theorem}
%%%%%%%%%%%%%%%%%%%%%%%%%%%%%%%%%%%%%%%%%%%%%%%%%%%%%%%%%%%%%%%%%
%%%    Proof of Theorem 2.2      %%%%%%%%%%%%%%%%%%%%%%%%%%%%%%%%
%%%%%%%%%%%%%%%%%%%%%%%%%%%%%%%%%%%%%%%%%%%%%%%%%%%%%%%%%%%%%%%%%
\begin{proof}
\quad Put $u(t) = ( u_1(t), u_2(t))$ and 
\begin{equation*}
   B(u) := \{ ~\! b \in H ~\! ; ~\!
             b = ( - u_1 u_2, - a  u_1) ~\! \}, 
\end{equation*}
then \eqref{1.2} can be reduced to the following abstract evolution equation in $H$: 
\begin{equation}\label{eq:abstract}
 \frac{d}{dt} u(t) + \partial \phi(u(t)) + B(u(t)) = 0, \quad u(0) = (u_{10}, u_{20}). 
\end{equation}
 We are going to apply Theorem II of \cite{O3}. To do this, we have to check 
 three assumptions. The compactness assumption (A.1) requires that 
 the set $\{ ~\! u \in H ~\! ; ~\! \phi(u) + |u|_H^2 \leq L ~\! \}$ is compact in $H$ for all $L>0$, 
 which is assured by the Rellich-Kondrachov theorem. The demiclosedness assumption (A.2) on 
 $B(u)$ is assured by the continuity of the mapping 
$ (u_1, u_2) \mapsto ( - u_1 u_2, - a  u_1)$ in $\mathbb R^2$. 
\\ 
 The last assumption to check is the boundedness assumption (A.4):
\begin{equation}\label{A3}
  | B(u) |_H^2 \leq k ~\! | \partial \phi(u)|_H^2 + \ell( \phi(u) + |u|_H ) 
     \quad \forall u \in D(\partial \phi), 
\end{equation}
  where $k \in [0,1)$ and $\ell(\cdot) : [0,\infty) \to [0,\infty)$ is a monotone increasing function. 
We note that 
\begin{equation}\label{est:B}
 |B(u)|_H^2 \leq \|u_1\|_4^2\|u_2\|_4^2 + a^2 \|u_1\|_2^2, 
   \quad \exists C>0 \ \text{such that} \ C ( \|u_1\|^2 + \|u_2\|^2 ) \leq \phi(u) + 1.
\end{equation}
  Hence for $N\leq4$, \eqref{A3} holds true with $k=0$ and $\ell(r) = C r^2$. 
  \\
As for the case where $N=5$, Gagliardo-Nirenberg interpolation inequality gives
\begin{equation*}
   \|v\|_4 \leq C \|v\|_{H^2}^{\frac{1}{4}} \|v\|^{\frac{3}{4}}.
\end{equation*}
Then by Young's inequality, \eqref{A3} is satisfied with $ \ell(r)= C r^3$. 
Thus the local existence part is verified. 
\\
To prove the uniqueness part, let $u^1=(u_1^1, u_2^1), \ u^2 = (u^2_1, u^2_2)$ be 
solutions of \eqref{1.2} and put $\delta u_i = u^1_i - u_i^2 \ (i=1,2).$ 
Then $\delta u_i$ satisfy
\begin{align}
  & \partial_t \delta u_1 - \Delta \delta u_1 + b \delta u_1
        =  \delta u_1 u_2^1 + \delta u_2 u_1^2, \label{eq1:delta}
\\
  & \partial_t \delta u_2 - \Delta \delta u_2 = a \delta u_1, \label{eq2:delta}
\\
   & \partial_{\nu} \delta u_1 + \alpha \delta u_1 
          = \partial_{\nu} \delta u_2 
          + \beta ( |u_2^1|^{\gamma-2}u_2^1 - |u_2^2|^{\gamma-2}u_2^2 ) = 0. \label{boundary:delta}
\end{align}
   Multiplying \eqref{eq1:delta} by $\delta u_1$ and \eqref{eq2:delta} by $\delta u_2$, 
we have by \eqref{boundary:delta}
\begin{align}
   & \frac{1}{2} \frac{d}{dt} \|\delta u_1(t)\|_2^2 + \| \nabla \delta u_1 \|_2^2 
            + \alpha \| \delta u_1 \|_{2,\partial\Omega}^2 + b \|\delta u_1 \|_2^2 
                \leq \int_\Omega ( | \delta u_1|^2 ~\! |u_2^1| 
                           + |\delta u_1 | ~\! | \delta u_2| ~\! |u_1^2| ) ~\! dx, 
     \label{est:delta1}
\\[2mm]
   &  \frac{1}{2} \frac{d}{dt} \|\delta u_2(t)\|_2^2 + \| \nabla \delta u_2 \|_2^2 
            + \beta \int_{\partial \Omega} \!\! ( |u_2^1|^{\gamma-2}u_2^1 
                          - |u_2^2|^{\gamma-2}u_2^2 ) ~\! \delta u_2 ~\! d\sigma 
                \leq a \int_\Omega |\delta u_1 | ~\! | \delta u_2| ~\! dx. 
     \label{est:delta2}
\end{align}
   Let $ N \leq 5$, then since $H^1(\Omega)$ and $H^2(\Omega)$ are embedded in $L^{\frac{10}{3}}(\Omega)$ and $L^{10}(\Omega)$ respectively, by Young's inequality 
we find that for any $\varepsilon>0$ there exists $C_\varepsilon>0$ such that 
\begin{align*}
   \int_\Omega |\delta u_i| ~\! |\delta u_j| ~\! |w| ~\! dx 
         & \leq C ~\! \|\delta u_i \| ~\! \|\delta u_j \|_2 ~\! \| w \|_{H^2(\Omega)}
\\
         & \leq \varepsilon ~\! ( \|\nabla \delta u_i\|_2^2 + \| \delta u_i \|_2^2 ) 
                   + C_\varepsilon \| \delta u_j\|_2^2 \|w \|_{H^2(\Omega)}^2. 
\end{align*} 
Hence, by adding \eqref{est:delta1} and \eqref{est:delta2}, we obtain 
\begin{equation*}
  \frac{d}{dt} ( \| \delta u_1(t)\|_2^2 +  \| \delta u_2(t)\|_2^2 ) 
     \leq C ( \|u_2^1\|_{H^2(\Omega)}^2 +  \|u_1^2\|_{H^2(\Omega)}^2 + 1 ) 
                          ~\! ( \|\delta u_1(t)\|_2^2 +  \|\delta u_2(t)\|_2^2 ),  
\end{equation*}
  Thus since $u_2^1, u_1^2 \in L^2(0,T;H^2(\Omega))$, the uniqueness follows from Gronwall's inequality. 
  \\
 The nonnegativity of solutions can be proved by exactly the same argument as in the proof of Theorem 3.1 in 
 \cite{KO}.
\end{proof}
%%%%%%%%%%%%%%%%%%%%%%%%%%%%%%%%%%%%%%%%%%%%%%%%%%%%%%%%%%%%%%%
%%%%   Main Result and proof    %%%%%%%%%%%%%%%%%%%%%%%%%%%%%%%
%%%%%%%%%%%%%%%%%%%%%%%%%%%%%%%%%%%%%%%%%%%%%%%%%%%%%%%%%%%%%%%
\section{Main result and proof}
\hspace{\parindent} 
In what follows we always consider the case where $\gamma =2$ 
 and we are concerned with global solutions of  \eqref{1.1}.
We put 
 \(H^1=\{(w_1,w_2)\in H^1(\Omega)\times H^1(\Omega) ~\! ; ~\! w_1,w_2\ge0, w_1,w_2\not\equiv0\}\) and  \(V=\{(w_1,w_2)\in L^{\infty}(\Omega)\times L^{\infty}(\Omega) ~\! ; ~\! w_1,w_2\ge0, w_1,w_2\not\equiv0\}\).
Our main theorem can be stated as follows.
%%%%%%%%%%%%%%%%%%%%%%%%%%%%%%%%%%%%%%%%%%%%%%%%%%%%%%%%%%%%%%%%%%%%
\begin{theorem}
	\label{main}
	Let \(N=2,3\) and \(\alpha\le2\beta\).
	Assume that \((u_{10}, u_{20})\in H^1\) and \((u_1,u_2)\) is the corresponding global solution 
 of \eqref{1.1} satisfying the same regularity given in Theorem \ref{LWPH1}.
	Then there exist constants \(M_i=M_i(\|u_1\|,\|u_2\|)>0\) \((i=1,2)\) such that
	\begin{equation}
	\label{2.1}
	\sup_{t\ge0} \|u_1(t)\| \le M_1,~~~ \sup_{t\ge0} \|u_2(t)\| \le M_2.
	\end{equation}
	Moreover if \((u_{10},u_{20})\in V\) and \((u_1,u_2)\) is the corresponding global solution 
 of \eqref{1.1} satisfying the same regularity given in Theorem \ref{LWPinfty}.   
   Then there exist constants \({M'}_i={M'}_i(\|u_{10}\|_{\infty}, \|u_{20}\|_{\infty})>0\) \((i=1,2)\) such that
	\begin{equation}
	\label{2.2}
	\sup_{t\ge0} \|u_1(t)\|_{\infty} \le {M'}_1,~~~ \sup_{t\ge0} \|u_2(t)\|_{\infty} \le {M'}_2.
	\end{equation}		
\end{theorem}
%%%%%%%%%%%%%%%%%%%%%%%%%%%%%%%%%%%%%%%%%%%%%%%%%%%%%%%%%%%%%%%%%%%%%%%%%%%%
We divide the proof into several steps.
 We first derive the \(L^1\)-estimate of the solutions.
In this step, we rely on the properties of the first eigenvalue and the corresponding eigenfunction of \(-\Delta\) with the Robin boundary conditions :
\begin{lemma}[\(\cite{EOE}\)]
	\label{phi1}
	Let \(\lambda_1\) and \(\varphi_1\) be the first eigenvalue and the corresponding eigenfunction for the problem:
	\begin{equation}
	\label{2.3}
	\left\{
	\begin{aligned}
	& -\Delta\varphi = \lambda\varphi, && x\in\Omega,\\
	& \partial_{\nu}\varphi + \gamma\varphi = 0, &&x\in\partial\Omega,
	\end{aligned}
	\right.
	\end{equation}
	where \(\Omega\) is smooth bounded domain in \(\mathbb{R}^N\) and \(\gamma>0\). Then \(\lambda_1>0\) and there exists a constant \(C_{\gamma}>0\) such that
	\begin{equation*}
	\varphi_1(x)\ge C_{\gamma}\hspace{10mm}x\in\overline{\Omega}.
	\end{equation*}
\end{lemma}

Actually, it is easy to see that \(\varphi_1>0\) in \(\Omega\) by the strong maximum principle as the same method for the eigenvalue problem with the Dirichlet Laplacian. 
Furthermore suppose that there exists \(x_0\in\partial\Omega\) such that \(\varphi_1(x_0)=0\). 
Then the boundary condition assures \(\partial_\nu\varphi_1(x_0)=-\gamma\varphi_1(x_0)=0\). 
On the other hand, we know \(\partial_\nu\varphi_1(x_0)<0\) by Hopf's strong maximum principle. 
This is contradiction, i.e., \(\varphi_1(x)>0\) on \(\overline{\Omega}\).

 The second step is to derive uniform $L^2$-estimates and third one is to 
  derive uniform $H^1$-estimates.
%\begin{equation*}
%\|u\|_p\le\|u\|_q^{\theta}\|u\|_r^{1-\theta},
%\end{equation*}
%where \(\frac{1}{p}=\frac{\theta}{q}+\frac{1-\theta}{r}\).
In the last step, we get uniform \(L^{\infty}\) bounds for global solutions of \(\eqref{1.1}\) 
 applying Moser's iteration scheme (see \(\cite{A1}\) and \(\cite{Na1}\)).
%
%\begin{proof}[Proof of \(H^1\) bounds]
\\[2mm]
%%%%%%%%%%%%%%%%%%%%%%%%%%%%%%%%%%%%%%%%%%%%%%%%%%%%%%%%%%%%%%%%%%%
%%%%%  Uniform estmates in L^1    %%%%%%%%%%%%%%%%%%%%%%%%%%%%%%%%%
%%%%%%%%%%%%%%%%%%%%%%%%%%%%%%%%%%%%%%%%%%%%%%%%%%%%%%%%%%%%%%%%%%%
 (1) \underline{$ \text{\it Uniform estimates in} \ L^1$}
\\
	%We first consider the uniformly boundedness of \(L^1\) norm of solutions. 
\quad   Let \(\lambda_1\) and \(\varphi_1\) be the first eigenvalue and the corresponding eigenfunction of \(\eqref{2.3}\) respectively.
	We here normalize \(\varphi_1\) so that \(\|\varphi_1\|_1=1\).
	Multiplying \(\varphi_1\) by the first and second equations of \(\eqref{1.1}\), we get 
\begin{align}
		& \Bigl( \int_{\Omega} u_1\varphi_1 dx \Bigr)_t 
                 + (b+\lambda_1) \int_{\Omega} u_1\varphi_1 dx  
                   +  (\alpha-\gamma) \int_{\partial\Omega} u_1\varphi_1 d\sigma 
                      = \int_{\Omega}u_1u_2\varphi_1 dx,
           \label{3.4}
\\[1.5mm]
		& \Bigl( \int_{\Omega} u_2\varphi_1 dx \Bigr)_t 
                + \lambda_1 \int_{\Omega} u_2\varphi_1 dx 
                   + (\beta-\gamma) \int_{\partial\Omega} u_2\varphi_1 d\sigma 
                      = a \int_{\Omega} u_1\varphi_1 dx. 
			\label{3.5}	
\end{align}
	Multiplying \eqref{3.4} by $a$ and substituting \eqref{3.5} and equation \eqref{1.1}  
to the second term of the left-hand side and the right-hand side respectively, we have  
\begin{equation}\label{eq:34:deform} 
   \begin{split}
   a \Bigl( \int_{\Omega} u_1\varphi_1 dx \Bigr)_t 
      + (b + \lambda_1) \Bigl( \Bigl( \int_{\Omega} u_2\varphi_1 dx \Bigr)_t 
             & + \lambda_1 \int_{\Omega} u_2\varphi_1 dx 
                + (\beta-\gamma) \int_{\partial\Omega} u_2\varphi_1 d\sigma  \Bigr)
\\[2mm]
     & + (\alpha-\gamma) \int_{\partial\Omega} u_1\varphi_1 d\sigma
          = \int_{\Omega} \left( \partial_t u_2 - \Delta u_2 \right)u_2\varphi_1 dx
   \end{split}
\end{equation}	
	Then differentiating \eqref{3.5} with respect to $t$ once 
      and substituting \eqref{eq:34:deform} to the right-hand side, we obtain
\begin{align}
	\nonumber
		& \Bigl( \int_{\Omega} u_2\varphi_1 dx \Bigr)_{tt} 
           + (b+2\lambda_1) \Bigl( \int_{\Omega} u_2\varphi_1 dx \Bigr)_t 
              + \lambda_1(b+\lambda_1) \int_{\Omega} u_2\varphi_1 dx
\\[1.5mm]
		\nonumber
		& \hspace{7mm} + (\alpha-\gamma) \int_{\partial\Omega} u_1\varphi_1 d\sigma 
           + (\beta-\gamma) \Bigl( \int_{\partial\Omega} u_2\varphi_1 d\sigma \Bigr)_t 
              + (\beta-\gamma)(b+\lambda_1) \int_{\partial\Omega} u_2\varphi_1 d\sigma
\\[1.5mm]
		= &~ \int_{\Omega} \left( \partial_t u_2 - \Delta u_2 \right)u_2\varphi_1 dx 
\nonumber
\\[1.5mm]
		= &~ \frac{1}{2}  \Bigl( \int_{\Omega} u_2^2\varphi_1 dx \Bigr)_t 
               + \int_{\Omega} |\nabla u_2|^2\varphi_1 dx 
                  + \frac{\lambda_1}{2} \int_{\Omega} u_2^2\varphi_1 dx 
                     + \left( \beta - \frac{\gamma}{2} \right) 
                         \int_{\partial\Omega} u_2^2\varphi_1 d\sigma.
\end{align}
	Finally choosing \(\gamma=\frac{\alpha+2\beta}{2}>0\), we deduce  
\begin{equation}\label{3.8}
   \begin{split}
	 \Bigl( \int_{\Omega} u_2\varphi_1 dx \Bigr)_{tt} 
        & + (b+2\lambda_1) \Bigl( \int_{\Omega} u_2\varphi_1dx \Bigr)_{t} 
           + \lambda_1(b+\lambda_1) \int_{\Omega}u_2\varphi_1 dx
\\[2mm]
    &  - \frac{\alpha}{2}\Bigl( \int_{\partial\Omega} u_2\varphi_1d\sigma \Bigr)_t
       - \frac{\alpha}{2}\lambda_1\int_{\partial\Omega} u_2\varphi_1d\sigma 
 \geq \frac{1}{2} \Bigl( \int_{\Omega} u_2^2\varphi_1 dx \Bigr)_t 
                            + \frac{\lambda_1}{2} \int_{\Omega} u_2^2\varphi_1 dx. 
   \end{split}
\end{equation}
We now set 
\begin{equation*}
  y(t): = w'(t) + ( b + \lambda_1)~\! w(t) 
            - \frac{1}{2}\int_{\Omega}u_2^2 ~\! \varphi_1 ~\! dx
              - \frac{\alpha}{2}\int_{\partial\Omega} \!\! u_2 ~\! \varphi_1 ~\! d\sigma, 
   \quad w(t) := \int_{\Omega}u_2 ~\! \varphi_1 ~\! dx.
\end{equation*}
  Since $\partial_t u_2 \in L^2(0,T;L^2(\Omega))$ implies that 
  there exists $s_0 \in (0,1)$ such that $ |y(s_0)| < \infty$.
Then \eqref{3.8} yields 
\begin{equation*}
  y'(t) \geq - \lambda_1 ~\! y(t), \quad \text{hence} \quad  
         y(t) \geq y(s_0) ~\! e^{-\lambda_1 (t - s_0)} \geq - | y(s_0)| =:  - C_0 
              \quad \forall t \geq s_0.
\end{equation*}
   Hence by virtue of Schwarz's inequality and Young's inequality, we get 
\begin{align*}
	-C_0 \le y(t) & =  w'(t) + (b+\lambda_1) ~\! w(t) 
                       - \frac{1}{2}\int_{\Omega}u_2^2\varphi_1 ~\! dx 
                         -\frac{\alpha}{2}\int_{\partial\Omega} \!\! u_2 \varphi_1 ~\! d\sigma
\\[1.5mm]
        	& \le w'(t) + (b+\lambda_1) ~\! w(t) - \frac{1}{2}w^2(t)
\\[1.5mm]
        	 & \le w'(t) - \frac{1}{4} ~\! w^2(t) + (b+\lambda_1)^2
                  \qquad \forall t \geq s_0,
\end{align*}
i.e.,
\begin{equation}\label{3.9}
	w'(t) \ge \frac{1}{4}w^2(t)-C_1, \quad C_1:=C_0+(b+\lambda_1)^2>0 \quad \forall t \ge s_0, 
\end{equation}
  whence follows
\begin{equation}\label{3.10}
	w(t) \le 2C_1^{\frac{1}{2}} =: C_2 \qquad  \forall t\geq s_0, 
\end{equation}
Indeed, if there exists \(t_1\ge s_0\) such that 
\begin{equation}\label{3.11}
	\frac{1}{4}w^2(t_1)-C_1>0,
\end{equation}
then from \(\eqref{3.9}\), \(\eqref{3.11}\) we can deduce that there exists \(t_2>t_1\) such that 
\begin{equation*}
	\lim_{t\to t_2}w(t) = + \infty, 
\end{equation*}
which contradicts the assumption that \(w(t)\) exists globally.
Thus \(\eqref{3.10}\) holds and the following global bound for $w(t)$ is established. 
\begin{equation}\label{3.12}
	\sup_{t\ge 0}\int_{\Omega} u_2 ~\! \varphi_1 ~\! dx 
             \leq \overline{C}_2 := \max \Bigl( C_2, \max_{0\leq s \leq s_0} w(s) \Bigr).
\end{equation}
%%%%%%%%%%%%%%%%%%%%%%%%%%%%%%%%%%%%%%%%%%%%%%%%%%%%%%%%%%%%%%
%%%%   Estimate for u_1     %%%%%%%%%%%%%%%%%%%%%%%%%%%%%%%%%%
%%%%%%%%%%%%%%%%%%%%%%%%%%%%%%%%%%%%%%%%%%%%%%%%%%%%%%%%%%%%%%
\quad Next we derive a uniform estimate for \(\int_{\Omega}u_1\varphi_1dx\). 
Using the facts that \(u_1=\frac{1}{a}(\partial_t u_2-\Delta u_2)\) and \((u_1,u_2)\) are nonnegative in \eqref{3.4}, we can get
\begin{align*}
    \frac{d}{dt}\Bigl( \int_{\Omega}u_1\varphi_1 dx \Bigr) 
     & \ge - (b+\lambda_1) \int_{\Omega}u_1\varphi_1 ~\! dx 
\\[1.5mm]
     & = - (b+\lambda_1)\frac{1}{a}\int_{\Omega}(\partial_t u_2 - \Delta u_2)\varphi_1 ~\! dx
\\[1.5mm]
     & = - \frac{b+\lambda_1}{a}w'(t) - \frac{(b+\lambda_1)\lambda_1}{a}w(t) 
             + \frac{(b+\lambda_1)\alpha}{2a}\int_{\partial\Omega}u_2\varphi_1 d\sigma
\\[1.5mm]
     & \ge - \frac{b+\lambda_1}{a}w'(t) - \frac{(b+\lambda_1)\lambda_1}{a}w(t). 
\end{align*}
%%%%%%%%%%%%%%%%%%%%%%
 For \(\eta\in(0,1)\), integrating this inequality over \((t,t+\eta)\) 
   and using \(\eqref{3.12}\), we obtain
\begin{align*}
   \left[ \int_{\Omega} u_1\varphi_1 dx \right]^{t+\eta}_{t} 
     & \ge - \frac{b+\lambda_1}{a} \left( w(t+\eta) - w(t) \right) 
              - \frac{(b+\lambda_1)\lambda_1}{a}\int_{t}^{t+\eta} \!\! w(\tau) ~\! d\tau
\\[1.5mm]
     & \ge - \frac{b+\lambda_1}{a} ~\! \overline{C}_2 
              - \frac{(b+\lambda_1)\lambda_1}{a} ~\! \overline{C}_2 =:-C_3,
\end{align*}
where \(C_3>0\) is independent of \(t\) and \(\eta\).
This implies that
\begin{equation}\label{3.13}
   \int_{\Omega}u_1(t)\varphi_1 ~\! dx 
          \le C_3 + \int_{\Omega} u_1(t+\eta)\varphi_1 ~\! dx.
\end{equation} 
   Integrating \(\eqref{3.13}\) over \(\eta\in(0,1)\) and using integration by parts, we get
\begin{align*}
   \int_{\Omega} u_1(t)\varphi_1 dx 
      & \le C_3 + \int_{0}^{1}\int_{\Omega} u_1(t+\eta)\varphi_1 ~\! dx ~\! d\eta
\\[1.2mm]
      &  = C_3 + \int_{t}^{t+1} \!\!\! \int_{\Omega} u_1(\tau)\varphi_1 ~\! dx ~\! d\tau
\\[1.2mm]
      &  = C_3 + \frac{1}{a}\int_{t}^{t+1} \!\!\!
                   \int_{\Omega} (\partial_t u_2 -\Delta u_2)\varphi_1 ~\! dx ~\! d\tau
\\[1.2mm]
      &  = C_3 + \frac{1}{a}\left( w(t+1) - w(t) \right) 
                     + \frac{\lambda_1}{a}\int_{t}^{t+1} \!\! w(\tau) ~\! d\tau 
                       - \frac{\alpha}{2a}\int_{t}^{t+1} \!\!\!
                            \int_{\partial\Omega}u_2\varphi_1 ~\! d\sigma ~\! d\tau
\\[1.2mm]
      &  \le C_3 + \frac{1+\lambda_1}{a} ~\! \overline{C}_2=:C_4,
\end{align*} 
which concludes that 
\begin{equation}\label{3.14}
\sup_{t\ge 0}\int_{\Omega}u_1\varphi_1dx\le C_4.
\end{equation}
Thus, from \(\eqref{3.12}\), \(\eqref{3.14}\) and Lemma \(\ref{phi1}\), we can derive the following estimates:
\begin{equation}\label{3.15}
\sup_{t\ge 0}\|u_1(t)\|_{1}\le C_5,~~~\sup_{t\ge 0}\|u_2(t)\|_{1}\le C_6.
\end{equation}
%where \(C_5\), \(C_6>0\) do not depend on time \(t\).
\\[2mm]
%%%%%%%%%%%%%%%%%%%%%%%%%%%%%%%%%%%%%%%%%%%%%%%%%%%%%%%%%%%%%%%%%%%
%%%%%%%%%%%  Uniform estimate in L^2    %%%%%%%%%%%%%%%%%%%%%%%%%%%
%%%%%%%%%%%%%%%%%%%%%%%%%%%%%%%%%%%%%%%%%%%%%%%%%%%%%%%%%%%%%%%%%%%
 (2) \underline{$ \text{\it Uniform estimates in} \ L^2$}
\\[2mm]
 \quad We here try to get \(L^2\) uniform bounds of solutions of \(\eqref{1.1}\). 
Since \(\eqref{3.4}\) gives 
\begin{equation*}
   \int_{\Omega} u_1u_2\varphi_1 dx
       \le \frac{d}{dt}\Bigl( \int_{\Omega} u_1\varphi_1 ~\!dx \Bigr) 
               + (b+\lambda_1) \int_{\Omega} u_1\varphi_1~\! dx, 
\end{equation*}
it follows from \eqref{3.14} that  
\begin{equation}\label{3.16}
    \sup_{t\ge 0} \int_{t}^{t+1} \!\!\! \int_{\Omega} u_1u_2  ~\! dx ~\! d\tau \le C_7.
\end{equation}
  Multiplying the second equation of \(\eqref{1.1}\) by \(u_2\) and using integration by parts, we get 
\begin{equation*}
   \frac{1}{2}\frac{d}{dt}\|u_2(t)\|_2^2 + \|\nabla u_2(t)\|_2^2
       + \beta\|u_2(t)\|_{2,\partial\Omega}^2 = a\int_{\Omega}u_1u_2 ~\! dx,
\end{equation*} 
where \(\|v\|_{2,\partial\Omega}^2=\int_{\partial\Omega} v^2 d\sigma\). 
  Hence by virtue of Poincar\'e - Friedrichs' inequality   
 \( C_F \|v\|_2^2 \le (\|\nabla v\|_2^2 + \beta\|v\|_{2,\partial\Omega}^2)\), 
   we have
\begin{equation}\label{3.17} 
	\frac{1}{2}\frac{d}{dt} \|u_2(t)\|_2^2 
       + C_F \|u_2(t)\|_2^2 
           \le a\int_{\Omega}u_1u_2 ~\! dx.
\end{equation}
Applying Gronwall's inequality to \(\eqref{3.17}\), we get
\begin{equation}\label{3.18}
	\|u_2(t)\|_2^2 \le e^{-2 C_F t} \|u_{20}\|_2^2 
          + \int_{0}^{t} 2a \Bigl( \int_{\Omega}u_1u_2 ~\! dx \Bigr) e^{-2 C_F(t-\tau)} ~\! d\tau.
\end{equation}
  In order to obtain uniform bounds of \(L^2\)-norm for \(u_2\) with respect to \(t\), 
   we need to confirm that the second term of right hand side of \(\eqref{3.18}\) is bounded. 
    For any \(t\ge 0\), we can express \(t=n+\varepsilon\) with some \(n\in\mathbb{N}\cup\{0\}\)     and \(\varepsilon\in[0,1)\).
       Then, by virtue of \(\eqref{3.16}\), it follows that
\begin{align*}
    & \int_{0}^{t}\Bigl( \int_{\Omega}u_1u_2 dx \Bigr) e^{-2 C_F (t-\tau)} ~\! d\tau
\\[1.5mm]
   = & \int_{t-1}^{t}\Bigl( \int_{\Omega}u_1u_2 ~\! dx \Bigr) e^{-2 C_F (t-\tau)} ~\! d\tau 
        + \int_{t-2}^{t-1}\Bigl( \int_{\Omega}u_1u_2 ~\! dx \Bigr) e^{-2 C_F (t-\tau)} ~\! d\tau
\\[1.5mm]
     & + \cdots + \int_{t-n}^{t-(n-1)}\Bigl( \int_{\Omega}u_1u_2 ~\! dx \Bigr) e^{-2 C_F (t-\tau)} ~\! d\tau 
         + \int_{0}^{t-n}\Bigl( \int_{\Omega}u_1u_2 ~\! dx \Bigr) e^{-2 C_F (t-\tau)}~\!  d\tau
\\[1.5mm]
    \le & ~ e^{-0} \int_{t-1}^{t}\Bigl( \int_{\Omega}u_1u_2 ~\! dx \Bigr) d\tau 
             + e^{-2 C_F} \int_{t-2}^{t-1}\Bigl( \int_{\Omega}u_1u_2 ~\! dx \Bigr) d\tau
\\[1.5mm]
       &  + \cdots + e^{-2(n-1) C_F} \int_{t-n}^{t-(n-1)}\Bigl( \int_{\Omega}u_1u_2 ~\! dx \Bigr)d\tau 
           + e^{-2n C_F} \int_{0}^{t-n}\Bigl( \int_{\Omega}u_1u_2 ~\! dx \Bigr) d\tau
\\[1.5mm]
      \le & ~C_7 \Bigl( 1 + e^{-2 C_F} + e^{-4 C_F} + \cdots + e^{-2n C_F} \Bigr)
\\[1.5mm]
       =  & ~C_7 \frac{1-e^{-2(n+1) C_F}}{1-e^{-2 C_F}} \leq \frac{C_7}{1-e^{-2 C_F}}.
\end{align*}
  Therefore we obtain from \(\eqref{3.18}\)
\begin{equation*}
    \|u_2(t)\|_2^2 \le e^{-2 C_F t} \|u_{20}\|_2^2 + \frac{2 a C_7}{1-e^{-2 C_F}} \qquad \forall t\ge 0.
\end{equation*}
  This implies that there exists \(C_8>0\) such that
\begin{equation}\label{3.19}
	\sup_{t\ge 0}\|u_2(t)\|_2 \le C_8.
\end{equation}
   Note that the above argument can be done without any restriction on dimension \(N\).  
\\
%%%%%%%%%%%%%%%%%%%%%%%%%%%%%%%%%%%%%%%%%%%%%%%%%%%%%%%%%%%%%%%%%%%%%%%%%%%%%%%%%%%%%%%%%%%%
%%%%%%%   L^2 estimate for u_1     %%%%%%%%%%%%%%%%%%%%%%%%%%%%%%%%%%%%%%%%%%%%%%%%%%%%%%%%%
%%%%%%%%%%%%%%%%%%%%%%%%%%%%%%%%%%%%%%%%%%%%%%%%%%%%%%%%%%%%%%%%%%%%%%%%%%%%%%%%%%%%%%%%%%%%
    We next derive a uniform \(L^2\)-estimate of \(u_1\) for $ N \leq 3$.
%We first deal with the case \(N=3\).
  Multiplying the first equation of \(\eqref{1.1}\) by \(u_1\) and using integrating by parts, we have
\begin{equation*}
      \frac{1}{2}\frac{d}{dt}\|u_1(t)\|_2^2 + \|\nabla u_1(t)\|_2^2 
           + \alpha \|u_1(t)\|_{2,\partial\Omega}^2 + b\|u_1(t)\|_2^2
              = \int_{\Omega} u_1^2u_2 ~\! dx.
\end{equation*}
 We here adopt \((\|\nabla v\|_2^2  + b ~\! \|v\|_2^2)^{1/2} \) as the \(H^1\) norm for \(u_1\). 
%We note that this norm gives the usual \(H^1\) norm when \(\alpha=0\).
   By using H\"older's inequality, the interpolation inequality and the embedding theorem
     (\(\|v\|_6\le C_9\|v\|\)), it holds that
\begin{align*}%\label{3.20}
     \frac{1}{2}\frac{d}{dt}\|u_1(t)\|_2^2 + \|u_1(t)\|^2
       & \leq \int_{\Omega} u_1^2u_2 ~\! dx
\\[1.5mm]%\nonumber
       & \le \|u_1(t)\|_4^2 \|u_2(t)\|_2
\\[1.5mm]%\nonumber
       & \le \|u_1(t)\|_1^{\frac{1}{5}} \|u_1(t)\|_{6}^{\frac{9}{5}} \|u_2(t)\|_2
\\[1.5mm]%\nonumber
       & \le C_5^{\frac{1}{5}} C_8 C_{9}^{\frac{9}{5}} \|u_1(t)\|^{\frac{9}{5}}\le \frac{1}{2}\|u_1(t)\|^2 + C_{10}, 
\end{align*} 
which implies 
\begin{equation*}
  \frac{1}{2}\frac{d}{dt}\|u_1(t)\|_2^2 + \frac{1}{2}\|u_1(t)\|^2  \le C_{10}.
\end{equation*}
 Hence we obtain 
\begin{equation*}
     \|u_1(t)\|_2^2  \le e^{-t}\|u_{10}\|_2^2 + 2C_{10}\left( 1-e^{-t} \right),
\end{equation*}
i.e.,
\begin{equation}\label{3.20}
\sup_{t\ge 0}\|u_1(t)\|_2 \le C_{11}.
\end{equation}
%%%%%%%%%%%%%%%%%%%%%%%%%%%%%%%%%%%%%%%%%%%%%%%%%%%%%%%%%%%%%%%
%%%%%%%%%%%  Uniform H^1 estimates   %%%%%%%%%%%%%%%%%%%%%%%%%%
%%%%%%%%%%%%%%%%%%%%%%%%%%%%%%%%%%%%%%%%%%%%%%%%%%%%%%%%%%%%%%%
(3) \underline{$ \text{\it Uniform estimates in} \ H^1$}
\\[2mm]
\quad 
   Now we are in the position to derive a uniform \(H^1\) bounds of solutions of \eqref{1.1}.
Multiplying the second equation of \(\eqref{1.1}\) by \(-\Delta u_2\), we obtain  
\begin{align*}
   \frac{1}{2}\frac{d}{dt} ( \|\nabla u_2(t)\|_2^2 + \beta \|u_2(t)\|_{2,\partial \Omega}^2 )
       + \|\Delta u_2(t)\|_2^2 = - a \int_{\Omega}u_1\Delta u_2 ~\! dx 
            \le \frac{1}{2}\|\Delta u_2(t)\|_2^2 + \frac{a^2}{2}\|u_1(t)\|_2^2.
\end{align*}
  Here we define the $H^1$-norm of $u_2$ by
\begin{equation*}
    \| u_2 \|^2 := \|\nabla u_2(t)\|_2^2 + \beta \|u_2(t)\|_{2,\partial \Omega}^2. 
\end{equation*}
  Then it holds that \( C_F \| u_2 \|^2 \leq \|\Delta u_2\|_2^2 \), since
\begin{equation*}
   (C_F)^{\frac{1}{2}}  \|u_2\|_2 ~\! \|u_2 \| \le \|\nabla u_2\|_2^2 +  \beta \|u_2(t)\|_{2,\partial \Omega}^2 
          = (-\Delta u_2,u_2) \le \|\Delta u_2\|_2 \|u_2\|_2,
\end{equation*}
where \((\cdot,\cdot)\) denotes the inner product of \(L^2\).
  Hence we obtain
\begin{equation*}
    \frac{d}{dt}\|u_2(t)\|^2 + C_F \|u_2(t)\|^2 \le a^2 ~\! C_{11}^2,
\end{equation*}
whence follows
\begin{equation}\label{3.21}
\sup_{t\ge 0} \|u_2(t)\| \le C_{12}.
\end{equation}
%%%%%%%%%%%%%%%%%%%%%%%%%%%%%%%%%%%%%%%%%%%%%%%%%%%%%%%%%%%%%%%%
%%%%%%%   Unifrom H^1 estimate for u_1    %%%%%%%%%%%%%%%%%%%%%%
%%%%%%%%%%%%%%%%%%%%%%%%%%%%%%%%%%%%%%%%%%%%%%%%%%%%%%%%%%%%%%%%
 \quad In order to derive the uniform $H^1$-estimate for $u_1$, we prepare 
 the following functional $\phi_1(u_1)$:
 \begin{equation*}
    \phi_1(u_1) := \frac{1}{2} ( \|\nabla u_1 \|_2^2 + \alpha ~\! \|u_1\|_{2, \partial \Omega}^2 + 
                                    b ~\! \|u_1\|_2^2 ) \qquad u_1 \in H^1(\Omega).
 \end{equation*}
 Then it is easy to see
\begin{align}
  & \phi_1(u_1) \geq \frac{1}{2} \|u_1\|^2 \geq \frac{b}{2} \|u_1\|_2^2,
     \label{est:phi1}
\\[2mm]
  & \| - \Delta u_1 + b ~\! u_1 \|_2 \|u_1 \|_2 \geq | ( - \Delta u_1 + b ~\! u_1, u_1 ) | 
                                   = 2 ~\! \phi_1(u_1) 
                                     \geq 2 ~\! \sqrt{\phi_1(u_1)} \sqrt{\frac{b}{2}} \|u_1\|_2, 
   \notag
\end{align} 
  whence follows
\begin{equation}\label{est:Deltau1}
   2b ~\! \phi_1(u_1) \leq \| - \Delta u_1 + b ~\! u_1 \|_2^2.
\end{equation}
  Multiplication of the first equation of \(\eqref{1.1}\) by \( - \Delta u_1 + b u_1\) 
    and integration over \(\Omega\) yield 
\begin{equation}\label{3.24}
     (\partial_t u_1, - \Delta u_1 + b ~\! u_1) + \|- \Delta u_1 + b ~\! u_1 \|_2^2
          = ( u_1 u_2 , - \Delta u_1 + b ~\! u_1 )
               \leq  \frac{1}{2} ( \| u_1 u_2 \|_2^2 + \|- \Delta u_1 + b ~\! u_1 \|_2^2) .
\end{equation}
  Here we note 
\begin{equation*}
(\partial_t u_1, -\Delta u_1 + b ~\! u_1) = \frac{d}{dt} \phi_1(u_1(t)).
\end{equation*} 
   Hence, in view of \eqref{3.24} and \eqref{est:Deltau1}, we obtain 
\begin{equation*}
   \frac{d}{dt} \phi_1(u_1(t)) + b ~\! \phi_1(u_1(t)) \leq \frac{1}{2} \| u_1 u_2 \|_2^2.
\end{equation*}
   Here by H\"older's inequality, \eqref{3.19}, \eqref{3.20}, \eqref{3.21},\eqref{est:phi1} and Young's inequality, we get 
\begin{align*}
  \| u_1 u_2 \|_2^2 = \int_{\Omega}u_1^2 ~\! u_2^2 ~\! dx 
      = & \int_{\Omega} u_1^{\frac{1}{2}} u_2^{\frac{1}{2}} u_1^{\frac{3}{2}} u_2^{\frac{3}{2}} ~\! dx 
\\[1.5mm]
      \le & \Bigl( \int_{\Omega}u_1 u_2 ~\! dx \Bigr)^\frac{1}{2}
              \Bigl( \int_{\Omega}u_1^3 u_2^3 ~\! dx \Bigr)^\frac{1}{2} 
\\[1.5mm]
      \le &~ C_{11}^\frac{1}{2} C_{8}^\frac{1}{2} \|u_1(t)\|_6^\frac{3}{2} \|u_2(t)\|_6^\frac{3}{2}
\\[1.5mm]
%    \le &~ C_{11}^\frac{1}{2} C_{8}^\frac{1}{2} C_{9}^3 C_{12}^\frac{3}{2} \|u_1(t)\|^\frac{3}{2} 
      \le &~ b~\!  \phi_1(u_1(t)) + C_{13}.
\end{align*}
Hence it follows that
\begin{align*}
    \frac{d}{dt} \phi_1(u_1(t)) +  \frac{b}{2} \phi_1(u_1(t)) 
          \le \frac{C_{13}}{2}.
\end{align*}
Therefore, applying Gronwall's inequality, we deduce
\begin{equation*}
   \phi_1(u_1(t)) \le  \phi_1(u_1(0)) ~\!  e^{-\frac{b}{2} t}  +  \frac{C_{13}}{b}.
\end{equation*}
which implies that
\begin{equation}\label{3.25}
     \sup_{t\ge 0} \|u_1(t)\| \le C_{14}.
\end{equation}
%\end{proof}
%%%%%%%%%%%%%%%%%%%%%%%%%%%%%%%%%%%%%%%%%%%%%%%%%%%%%%%%%%%%%%%%
%%%%%%   Uniform L^\infty estimates     %%%%%%%%%%%%%%%%%%%%%%%%
%%%%%%%%%%%%%%%%%%%%%%%%%%%%%%%%%%%%%%%%%%%%%%%%%%%%%%%%%%%%%%%%
(4) \underline{$ \text{\it Uniform estimates in} \ L^\infty$}
\\[2mm]
\quad 
  Since Theorem \ref{LWPinfty} assures that there exists $s_1 \in (0,1)$ such that 
   $u(s_1) \in H^1(\Omega)$ and $\|u(t)\|_\infty$ is bounded on $[0,s_1]$, 
    we can assume without loss of generality that \((u_{10},u_{20})\in H^1 \cap V\). 
To derive \(L^{\infty}\) bounds via \(H^1\) bounds, we rely on the following Alikakos - Moser's iteration scheme,
 which plays an essential role in our argument.

%%%%%%%%%%%%%%%%%%%%%%%%%%%%%%%%%%%%%%%%%%%%%%%%%%%%%%%%%%%%%%%%%%%%%%%%%%%%%%%%%%%%%%%%%%%%%
%%%%%%%%%%   Lemma 3.3     %%%%%%%%%%%%%%%%%%%%%%%%%%%%%%%%%%%%%%%%%%%%%%%%%%%%%%%%%%%%%%%%%%
%%%%%%%%%%%%%%%%%%%%%%%%%%%%%%%%%%%%%%%%%%%%%%%%%%%%%%%%%%%%%%%%%%%%%%%%%%%%%%%%%%%%%%%%%%%%%
\begin{lemma}[\(\cite{Na1}\)]\label{Moser}
	Assume that \(v\in W^{1,2}_{loc}([0,\infty);L^2(\Omega))\cap L^{\infty}_{loc}([0,\infty);L^{\infty}(\Omega)\cap H^1(\Omega))\) satisfies
	\begin{equation}\label{3.26}
	\frac{d}{dt}\|v(t)\|_r^r + c_1 r^{-\theta_1} \| |v(t)|^{\frac{r}{2}} \|^2 \le c_2r^{\theta_2}\left( \|v(t)\|_r^r + 1 \right) \hspace{10mm} a.e. ~t\in[0,\infty),
	\end{equation}
	for all \(r\in[2,\infty)\), where \(c_1>0\) and \(c_2\), \(\theta_1\), \(\theta_2\ge0\).
	Then there exist some constants \(d_1\), \(d_2\), \(d_3\) and \(d_4\ge0\) such that
	\begin{equation*}
		\sup_{t\ge 0} \|v(t)\|_{\infty} \le d_1 2^{\theta_2+ (\theta_1+\theta_2)d_2} M_0,
	\end{equation*}
	where \(M_0=\max(1, d_3\|v_0\|_{\infty},\sup_{t\ge 0}\|v(t)\|_2^{d_4})\).
\end{lemma}
%%%%%%%%%%%%%%%%%%%%%%%%%%%%%%%%%%%%%%%%%%%%%%%%%%%%%%%%%%%%%%%%%%%%%%%%%%%%%%%%%%%%%%%%%%%
%\begin{proof}[Proof of \(L^{\infty}\) bounds]
	In order to apply Lemma \(\ref{Moser}\), we deform \eqref{1.1} in the following way: 
	\begin{equation}\label{3.27}
		\partial_t u_1 - \Delta u_1 + u_1 = u_1 u_2 - b ~\! u_1 + u_1,
	\end{equation}
	\begin{equation}\label{3.28}
		\partial_t u_2 - \Delta u_2 + u_2 = a ~\! u_1 + u_2.
	\end{equation}
	Hereafter we employ the usual \(H^1\) norm \((\|\nabla v\|_2^2 + \|v\|_2^2)^{1/2}\) for \(u_1\) and \(u_2\).
	 Multiplying \(\eqref{3.27}\) by \(|u_1|^{r-2}u_1\) (\(r\ge 2\)) and using integration by parts, we obtain
	\begin{align*}
   & \frac{1}{r}\frac{d}{dt} \|u_1(t)\|_r^r + (r-1)\int_{\Omega}|\nabla u_1|^2 |u_1|^{r-2}~\! dx 
                   + \int_{\partial\Omega} |u_1|^r ~\! d\sigma + \|u_1(t)\|_r^r
\\[1.5mm]
   & = \int_{\Omega} u_1^r u_2 ~\! dx - b ~\! \|u_1(t)\|_r^r + \|u_1(t)\|_r^r.
	\end{align*}
 Hence we have 
	\begin{equation*}
     \frac{1}{r}\frac{d}{dt} \|u_1(t)\|_r^r 
        + (r-1)\int_{\Omega}|\nabla u_1|^2 |u_1|^{r-2} ~\! dx 
            + \|u_1(t)\|_r^r
		        \le  \int_{\Omega} |u_1|^r |u_2| ~\! dx + \|u_1(t)\|_r^r.
	\end{equation*} 
  Moreover we note 
	\begin{align*}
		(r-1)\int_{\Omega}|\nabla u_1|^2 |u_1|^{r-2} ~\!  dx + \|u_1(t)\|_r^r 
          & = \frac{4(r-1)}{r^2}\int_{\Omega} \bigl|\nabla |u|^{\frac{r}{2}} \bigr|^2 ~\!dx 
                 + \| ~\! |u_1(t)|^{\frac{r}{2}}\|_2^2
\\[1.5mm]
		 & \ge \frac{4(r-1)}{r^2} \|~\! |u_1(t)|^{\frac{r}{2}}\|^2,
	\end{align*}
	where we used the fact that \(r\ge2\) implies \(\frac{4(r-1)}{r^2}\in(0,1]\) to the last inequality.
  Hence we obtain 
	\begin{equation}\label{3.29}
	\frac{1}{r}\frac{d}{dt} \|u_1(t)\|_r^r + \frac{4(r-1)}{r^2} \||u_1(t)|^{\frac{r}{2}}\|^2
		  \le  \int_{\Omega} |u_1|^r |u_2| ~\!  dx + \|u_1(t)\|_r^r.
	\end{equation}
	By using H\"older's inequality, interpolation inequality, 
      Sobolev's embedding theorem and Young's inequality, we can get
	\begin{align*}
	 \int_{\Omega} |u_1|^r |u_2| ~\! dx 
          & \le \|u_1(t)\|^r_{\frac{3r}{2}} ~\!  \|u_2(t)\|_3
\\[1.5mm]
		  & \le \|u_1(t)\|_r^{\frac{r}{2}} ~\!  \|u_1(t)\|_{3r}^{\frac{r}{2}} ~\! \|u_2(t)\|_3
\\[1.5mm]
		  & \le \|u_2(t)\|_3 ~\!  \|u_1(t)\|_r^{\frac{r}{2}} ~\! \| ~\! |u_1(t)|^{\frac{r}{2}}\|_6
\\[1.5mm]
		  & \le C_{15} ~\! \|u_1(t)\|_r^{\frac{r}{2}} ~\!  \| ~\! |u_1(t)|^{\frac{r}{2}}\| 
\\[1.5mm]
		  & \le \frac{2(r-1)}{r^2} \| ~\! |u_1(t)|^{\frac{r}{2}}\|^2 + \frac{C_{15}^2 ~\! r^2}{8(r-1)}\|u_1(t)\|_r^r.
	\end{align*} 
 Since \(r\ge 2\), it is easy to see that \(\frac{r^2}{8(r-1)}\le r\). 
	Then, from these observations, \(\eqref{3.29}\) leads to 
	\begin{equation*}
	 \frac{1}{r}\frac{d}{dt} \|u_1(t)\|_r^r + \frac{2(r-1)}{r^2} \| ~\! |u_1(t)|^{\frac{r}{2}}\|^2
		 \le C_{15}^2 ~\! r ~\! \|u_1(t)\|_r^r + \|u_1(t)\|_r^r,
	\end{equation*} 
	that is,
	\begin{equation}\label{3.30}
		\frac{d}{dt} \|u_1(t)\|_r^r + \| ~\! |u_1(t)|^{\frac{r}{2}}\|^2
		    \le C_{16} ~\! r^2 \bigl( \|u_1(t)\|_r^r + 1 \bigr).
	\end{equation}
	Here we used the fact that \(1\le\frac{2(r-1)}{r}\) provided that \(r\ge2\).
	Then \(u_1(t)\) satisfies \(\eqref{3.26}\) with \(c_1=1\), \(c_2=C_{16}\), \(\theta_1=0\) and \(\theta_2=2\).
	Thus applying Lemma \(\ref{Moser}\) to \(\eqref{3.30}\), we see that there exists \(C_{17}>0\) such that
	\begin{equation}\label{3.31}
	\sup_{t\ge 0} \|u_1(t)\|_{\infty} \le C_{17}.
	\end{equation}
	Finally, applying the same argument as above for \(u_2(t)\), we have
	\begin{equation}\label{3.32}
	\frac{1}{r}\frac{d}{dt}\|u_2(t)\|_r^r + \frac{4(r-1)}{r^2} \| ~\! |u_2(t)|^{\frac{r}{2}}\|^2 
         \le a \int_{\Omega}u_1u_2^{r-1}dx + \|u_2(t)\|_r^r.
	\end{equation}
	 Since \(\frac{r-1}{r}\le 1\) and \(\frac{1}{r}\le 1\), due to \(\eqref{3.31}\) we can deduce
	\begin{align*}
	  a \int_{\Omega}u_1u_2^{r-1} ~\! dx 
          & \le a ~\! C_{17} ~\! \|u_2(t)\|_{r-1}^{r-1}
\\[1.5mm]
	      & \le a  ~\! C_{17} \Bigl\{ \frac{r-1}{r}\|u_2(t)\|_r^r + \frac{1}{r}|\Omega| \Bigr\}
\\[1.5mm]
	      & \le a ~\! C_{17} \Bigl( \|u_2(t)\|_r^r + |\Omega| \Bigr),
	\end{align*}
	which implies 
	\begin{equation*}
	\frac{1}{r}\frac{d}{dt}\|u_2(t)\|_r^r 
        + \frac{4(r-1)}{r^2} \| ~\! |u_2(t)|^{\frac{r}{2}}\|^2 
            \le C_{18} \Bigl( \|u_2(t)\|_r^r+ 1 \Bigr),
	\end{equation*}
	for some \(C_{18}>0\).
	   Since \(2\le\frac{4(r-1)}{r}\), we conclude that
	\begin{equation}\label{3.33}
	\frac{d}{dt}\|u_2(t)\|_r^r + 2 ~\! \| ~\! |u_2(t)|^{\frac{r}{2}}\|^2 
           \le C_{18} ~\! r \Bigl(\|u_2(t)\|_r^r + 1 \Bigr).
	\end{equation}
	Then we can apply Lemma \(\ref{Moser}\) to \(\eqref{3.33}\) 
       with \(c_1=2\), \(c_2=C_{18}\), \(\theta_1=0\) and \(\theta_2=1\).
         Thus there exists \(C_{19}>0\) such that
	\begin{equation}\label{3.34}
	\sup_{t\ge 0} \|u_2(t)\|_{\infty} \le C_{19}.
	\end{equation}
	These a priori bounds \(\eqref{3.31}\) and \(\eqref{3.34}\) complete the proof.
\begin{flushright}
$\Box$
\end{flushright}
%\end{proof}

\end{document}